\documentclass[12pt]{article}
\usepackage{amsmath,color,graphicx, amssymb}
\usepackage{enumerate}
\usepackage{amsthm}
\usepackage{placeins}
\usepackage{fullpage}

  \theoremstyle{definition}
\newtheorem{remark}{Remark}
\newtheorem{conjecture}{Conjecture}

\newtheorem{theorem}{Theorem}
\newtheorem{lemma}{Lemma}

\newtheorem{corollary}{Corollary}
\makeatletter
\let\@@pmod\pmod

\DeclareMathOperator{\Aut}{Aut}
\DeclareRobustCommand{\pmod}{\@ifstar\@pmods\@@pmod}
\def\@pmods#1{\mkern4mu({\operator@font mod}\mkern 6mu#1)}
\makeatother
\allowdisplaybreaks
\begin{document}
\title{Average Cyclicity for Elliptic Curves in Torsion Families}
\author{Luke Fredericks}
\maketitle
\begin{abstract}
We prove asymptotic formulas for cyclicity of reductions of elliptic curves over the rationals in a family of curves having specified torsion. These results agree with established conditional results and with average results taken over larger families. As a key tool, we prove an analogue of a result of Vl\u{a}du\c{t} that estimates the number of elliptic curves over a finite field with a specified torsion point and cyclic group structure.
\end{abstract}

\section{Introduction}
Let $E$ be an elliptic curve over $\mathbb{Q}$. Then $E$ is defined by a Weierstrass equation

\begin{align*}
E:y^2+a_1xy+a_3y^2=x^3+a_2x^2+a_4x+a_6, \hspace{10pt} a_i\in \mathbb{Q}.
\end{align*}

By an admissible changes of variable and after clearing denominators, $E$ can be given by a short Weierstrass equation
\begin{align*}
E_{(a, b)}: y^2=x^3+ax+b, \hspace{5pt} a, b \in \mathbb{Z}.
\end{align*}

It is well-known that given an elliptic curve $E$ over a field $K$, the set $E(K)$ of $K$-rational points of $E$ has the structure of an abelian group. When $K$ is a number field, the Mordell-Weil theorem says that $E(K)$ is finitely generated, and Mazur's torsion theorem classifies the possibilities for $E(\mathbb{Q})_{tors}$. If $K=\mathbb{F}_q$ is the finite field of $q$ elements, then $E(\mathbb{F}_q)$ is a finite abelian group of rank at most 2; that is,
\begin{align*}
E(\mathbb{F}_q)\cong\mathbb{Z}/n_1 \mathbb{Z} \times \mathbb{Z}/n_2 \mathbb{Z}\text{ where }n_2 \mid n_1.
\end{align*} 
Denote by $a_q(E)=\#E(\mathbb{F}_q)-q-1$; then the Hasse bound states that
$  \mid a_q(E)\mid \leq 2\sqrt{q}.$

For all but finitely many primes $p$, we obtain an elliptic curve $E_p/\mathbb{F}_p$ by reducing the coefficients of a Weierstrass model of $E$ modulo $p$. It is natural to ask how how the properties of $E_p$ vary with $p$. Given $E/\mathbb{Q}$, denote by 
\begin{eqnarray*}
\pi_E^t(x)= \#\{p\leq x: a_p(E)=t\}\label{frobtrace}\\ 
\pi_E^{cyc}(x) = \#\{p\leq x: E_p(\mathbb{F}_p) \text{ is cyclic} \}\label{cyclic}. %\\ 
%\pi_E^{twin}(x)=\#\{p\leq x: \#E_p(\mathbb{F}_p) \text{ is prime}\}\label{koblitz} \\ 
%\pi_E^{ST}(x)= \#\{p\leq x: \alpha \leq\theta_p(E)\leq \beta \}\label{normalizedtrace}.
\end{eqnarray*}
%$\mathbb{Z}$ injects into the ring of endomorphisms of $E/\mathbb{Q}$ by sending $n\in \mathbb{Z}$ to the map $P\to nP$; if the ring of endomorphisms is strictly larger than $\mathbb{Z}$, we say that $E$ has \emph{complex multiplication}. The asymptotic behavior of $\pi_E^0$ was determined by Deuring. 

The asymptotic growth of these functions are the subjects of the Lang-Trotter conjecture and the cyclicity conjecture, respectively. 
\begin{conjecture}[Lang-Trotter]
Let $E/\mathbb{Q}$ be an elliptic curve, and let $r\in \mathbb{Z}$ with $r\neq 0$ if $E$ has complex multiplication. Then 
\begin{eqnarray} \label{langtrotter}
\pi_E^r(x)\sim C_{E, r}\frac{\sqrt{x}}{\log{x}}, 
\end{eqnarray}
where $C_{E, r}$ is an explicit constant depending only on $E$ and $r$.
\end{conjecture}

Denote by $\mathbb{Q}(E[k])$ the field obtained by adjoining to $\mathbb{Q}$ the coordinates of all points in the $k$-torsion subgroup $E[k]$ of $E$,  and by li$(x)$ the logarithmic integral.
\begin{conjecture}[Cyclicity]
Let $E/\mathbb{Q}$ be a non-CM elliptic curve. Then
\begin{align}\label{cyclicity}
\pi_E^{cyc}(x)\sim \left(\sum_{k\geq 1} \frac{\mu(k)}{[\mathbb{Q}(E[k]):\mathbb{Q})]}\right)\mathrm{li}(x).
\end{align}
\end{conjecture}

The study of $\pi_E^{cyc}$ goes back to the work of Borosh, Moreno and Porta \cite{boroshmorenoporta} who suggested that for certain chosen examples of $E/\mathbb{Q}$, $E_p(\mathbb{F}_p)$ is cyclic for infinitely many $p$. Serre formulated and proved the cyclicity conjecture conditional on the Generalized Riemann Hypothesis for the division fields of $E$ \cite{serreresume}. The best result to date is the following conditional theorem of Cojocaru and Murty.
 \begin{theorem}\cite[Theorem 1.1]{cojocarumurtycyclicity}  
Let $E$ be a non-CM elliptic curve defined over $\mathbb{Q}$ of conductor $N$. Assuming GRH for the Dedekind zeta functions of the division fields of $E$, we have that 
\begin{eqnarray}
\pi_E^{cyc}(x)=\mathfrak{f}_E \mathrm{li}(x)+O_N\left(x^{5/6}(\log{x})^{2/3}\right)
\end{eqnarray}
where
\begin{eqnarray}
 \mathfrak{f}_E=\sum_{k\geq 1} \frac{\mu(k)}{[\mathbb{Q}(E[k]):\mathbb{Q})]}.
\end{eqnarray}
\end{theorem}

Unconditionally, we cannot prove that the asymptotic in \eqref{langtrotter} or \eqref{cyclicity} holds for a single elliptic curve $E$. However, starting with the work of Fouvry and Murty \cite{fouvrymurtyMR1382477}, average versions of Conjectures \ref{langtrotter} and \ref{cyclicity} have been obtained. These average results provide strong unconditional evidence for the corresponding conjectures. 

 Let
\begin{eqnarray}
\pi_{1/2}(X)=\int_2^X\frac{dt}{2\sqrt{t} \log{t}}\sim\frac{\sqrt{X}}{\log{X}}.
\end{eqnarray}

In the case of the Lang-Trotter conjecture we have
\begin{theorem}\cite[Corollary 1.3]{davidpappalardi}  Let $E(a, b):y^2=x^3+ax+b$, and let $\epsilon>0$. If $A, B>X^{1+\epsilon}$ then we have as $X\to \infty$
\begin{eqnarray}
\frac{1}{4AB}\sum_{\substack{\mid a\mid \leq A\\ \mid b\mid \leq B}} \pi_{E(a, b)}^r(X) \sim D_r\pi_{1/2}(X)
\end{eqnarray}
where
\begin{eqnarray}
D_r=\frac{2}{\pi}\prod_{\ell\nmid r}\frac{\ell(\ell^2-\ell-1)}{(\ell+1)(\ell-1)^2}\prod_{\ell\mid r}\frac{\ell^2}{\ell^2-1}.
\end{eqnarray}

\end{theorem}
Here and throughout the paper, $\ell$ denotes a prime number, and a product over $\ell$ is taken over all primes satisfying the given conditions.

In the case of the cyclicity conjecture, Banks and Shparlinski proved

\begin{theorem}\cite[Theorem 17]{banksshparlinski} 
Let $\epsilon>0$ and $K>0$ be fixed. Then, for all integers $A$ and $B$ satisfying $AB\geq x^{1+\epsilon}, A, B \leq x^{1-\epsilon}$, we have 

\[\frac{1}{4AB}\sum_{|a|\leq A}\sum_{|b|\leq B} \pi_{E(a, b)}^{cyc}(x)=C_{cyc}\pi(x)+O(\pi(x)/(\log x)^K), \]
where 
\[C_{cyc}= \prod_{\ell \text{ prime }}\left(1-\frac{1}{\ell(\ell-1)(\ell^2-1)}\right), \] and the constant implied by $O$ depends only on $\epsilon$ and $K$.
\end{theorem}

The average asymptotics described above provide strong evidence for the conjectures in each case. In particular, the constants $\mathfrak{f}_E$ and $C_{cyc}$ and $C_{E, r}$ and $D_r$ are clearly closely related. We view $C_{cyc}$ and $D_r$ as idealized constants where the variation of individual curves are averaged out. Furthermore, Jones \cite{jonesaveragesof} proved that the average of the constants predicted by the respective conjectures is indeed the constant seen in the average results. 

Jones proof leveraged the fact (also due to Jones) that almost all elliptic curves are what are known as \emph{Serre curves} \cite{jonesserrecurves}. However, there are interesting families which consist entirely of elliptic curves that are \emph{not} Serre curves. These  curves are essentially invisible in the prior average results cited above; it is therefore of interest to study the averages of the functions $\pi_E^t$ and $\pi_E^{cyc}$ as $E$ varies over such a family. 

One class of such families are the torsion families -- the family of elliptic curves some specified torsion structure.
It follows from Mazur's theorem that if $E$ has a point of order $m\geq 2$, then $m\in \{2, 3, 4, 5, 6, 7, 8, 9, 10, 12\}$. The elliptic curves $E/\mathbb{Q}$ that have a rational point of order $m\geq 4$ lie in a one-parameter family; these were obtained by Kubert \cite{MR523268} and are given in Table \ref{tab:table-parameterizations}.

\FloatBarrier
\begin{table}[h]
\centering
\begin{tabular}{|c|c| }
  \hline
$m$ & $E_m(a)$  \\
\hline
$4$ &  $y^2 + x y -a y = x^{3} -a x^{2} $        \\
\hline
$5$ &  $y^2 + \left(  1-a\right) x y  -a y = x^{3} -a x^{2} $         \\
\hline
$6$ & $y^2 + \left(  1-a\right) x y - \left(a^{2} + a\right) y = x^{3} - \left(a^{2} + a\right) x^{2} $          \\
\hline
$7$ &  $y^2 + \left(1+ a-a^{2}  \right) x y + \left(a^{2}-a^{3} \right) y = x^{3} + \left(a^{2}-a^{3}  \right) x^{2}$        \\
\hline
$8$ &  $y^2 + \left(\frac{-2a^{2} + 4a - 1}{a}\right) x y + \left(-2a^{2} + 3a - 1\right) y = x^{3} + \left(-2a^{2} + 3a - 1\right) x^{2}$      \\
\hline
$9$ &  $y^2 + \left(1+ a^{2}-a^{3}   \right) x y + \left(a^{2}- 2a^{3}+ 2a^{4}-a^{5} \right) y = x^{3} + \left(a^{2}- 2a^{3}+ 2a^{4}-a^{5} \right) x^{2} $        \\
\hline
$10$ & $\begin{aligned}[t]y^2 +& \left(\frac{2a^{3} - 2a^{2} - 2a + 1}{a^{2} - 3a + 1}\right) x y + \left(\frac{-2a^{5} + 3a^{4} - a^{3}}{a^{4} - 6a^{3} + 11a^{2} - 6a + 1}\right) y \\ =& x^{3} + \left(\frac{-2a^{5} + 3a^{4} - a^{3}}{a^{4} - 6a^{3} + 11a^{2} - 6a + 1}\right) x^{2} \end{aligned}$ \\
\hline
$12$ & $\begin{aligned}[t]
y^2 +& \left(\frac{6a^{4} - 8a^{3} + 2a^{2} + 2a - 1}{a^{3} - 3a^{2} + 3a - 1}\right) x y + \left(\frac{-12a^{6} + 30a^{5} - 34a^{4} + 21a^{3} - 7a^{2} + a}{a^{4} - 4a^{3} + 6a^{2} - 4a + 1}\right) y \\ &= x^{3} + \left(\frac{-12a^{6} + 30a^{5} - 34a^{4} + 21a^{3} - 7a^{2} + a}{a^{4} - 4a^{3} + 6a^{2} - 4a + 1}\right) x^{2}
\end{aligned}$ \\
\hline
\end{tabular}
\caption{Parameterizations for elliptic curves with $m$-torsion. \label{tab:table-parameterizations}}
\end{table}
\FloatBarrier
The discriminant $\Delta_m(a)$ of the curve with $m$-torsion given above is 
\begin{eqnarray*}
\Delta_4(a) &=& (16a + 1) a^{4} \\
\Delta_5(a) &=& a^{5}  (a^{2} - 11a - 1)\\
\Delta_6(a) &=& (9a + 1)  (a + 1)^{3}  a^{6} \\
\Delta_7(a) &=& (a - 1)^{7}  a^{7}  (a^{3} - 8a^{2} + 5a + 1) \\
\Delta_8(a)&=& a^{-4}  (2a - 1)^{4}  (a - 1)^{8}  (8a^{2} - 8a + 1)\\
\Delta_9(a)&=&(a - 1)^9  a^9  (a^2 - a + 1)^3  (a^3 - 6a^2 + 3a + 1)\\
\Delta_{10}(a) &=& (2a - 1)^{5}  (a - 1)^{10}  a^{10} (a^{2}-3a + 1)^{-10} (4a^{2} - 2a - 1) \\
\Delta_{12}(a)&=& (a - 1)^{-24}  (2a - 1)^{6}  a^{12}  (6a^{2} - 6a + 1)  (2a^{2} - 2a + 1)^{3}  (3a^{2} - 3a + 1)^{4}.
\end{eqnarray*}

James \cite{jamesthreetorsion} gave the first results in this direction when he obtained an asymptotic for the Lang-Trotter conjecture on average over the family of curves with a point of order 3. Battista, Bayless, Ivnaov, and James \cite{battistabayless} extended this investigation over the family of elliptic curves which possess a $\mathbb{Q}$-rational point of order $m$ for $m=5, 7$ or $9$. They prove

\begin{theorem}\cite[Theorem 3]{battistabayless}  Let $E_m(a)$ be the parameterization of elliptic curves which have a rational point of order $m\in \{5, 7, 9\}$. Then for any $c>0$, we have
\begin{eqnarray*}
\frac{1}{\mathcal{C}(N)}\sideset{}{'}\sum_{|a|\leq N}\pi_{E_m(a)}^r(X)=\frac{2}{\pi}C_{r, m}\pi_{1/2}(X)+O\left(\frac{X^{3/2}}{N}+\frac{\sqrt{X}}{\log^cX}\right),
\end{eqnarray*}  
where $\sum'$ represents the sum over non-singular curves, $\mathcal{C}(N)$ represents the number of curves in the sum, and
\begin{eqnarray*}
C_{r,m}=C_r(m)\prod_{\substack{\ell\nmid m\\ \ell\nmid r}}\frac{\ell(\ell^2-\ell-1)}{(\ell+1)(\ell-1)^2}\prod_{\substack{\ell\nmid m\\ \ell\mid r}}\frac{\ell^2}{\ell^2-1}, 
\end{eqnarray*}
where
\begin{eqnarray*}
C_r(m)= \begin{cases} 	5/4 &\text{ if } m=5 \text{ and }r\equiv 0,3,4 \pmod{5},\\
7/6 &\text{ if } m=7 \text{ and }r\equiv 0,,3,4,5,6 \pmod{7},\\
3/2 &\text{ if } m=9 \text{ and }r\equiv 0,3,6 \pmod{9}.
						\end{cases}
\end{eqnarray*}

\end{theorem}

The main result in  this paper is to establish an average cyclicity result over torsion families of elliptic curves. More precisely, we prove 

\begin{theorem}\label{mainresult}
Let $\epsilon> 0$,  $A>x^{1+\epsilon}$. Let $E_m(a)$ denote the parameterization of elliptic curves over $\mathbb{Q}$ which have a rational $m$-torsion point for $m\neq 2, 3$. Then
\[\frac{1}{2A}\sideset{}{'}\sum_{\mid a\mid \leq A} \pi_{E_m(a)}^{cyc}(x) = C_m\prod_{\ell\nmid m}\left(1-\frac{1}{\ell(\ell-1)(\ell^2-1)}\right)\pi(x)  +O\left(\frac{x^{1-\epsilon}}{\log x}\right)
\]
where $C_m$ summarized in Table \ref{tab:table-cm}.
\begin{table}[h]
\centering
\begin{tabular}{| c| c| c| c| c| c| c| c| c| }
\hline
  $m$&$4$ & $5$  & $6$  & $7$  & $8$  & $9$  & $10$ &$12$ \\
\hline
 $C_m$&$\frac{1}{2}$ & $\frac{19}{20}$ &  $\frac{5}{12}$ & $\frac{41}{42}$ &  $\frac{1}{2}$ & $\frac{5}{6}$ & $\frac{19}{40}$ & $\frac{5}{12}$ \\
\hline
\end{tabular} 
\caption{\label{tab:table-cm}}
\end{table}
\end{theorem}
\begin{remark} The effect of the presence of $m$-torsion is apparent; we interpret the constant $C_{cyc}$ as a product of local factors, each of which is the probability that $E_p$ has cyclic $\ell$-torsion. The presence of a point of order $m$ should have some influence on these probabilities for $\ell\mid m$. Indeed, a curve with a point of order $\ell$ over $\mathbb{Q}$ need only acquire a single linearly independent point of order $\ell$ for cyclicity of the reduction to fail. Compared to the generic case of a curve without a point of order $\ell$, we expect it to be much less likely that the reductions of these curves are cyclic. Our result quantifies this heuristic reasoning.
\end{remark}

\begin{remark}
The family of curves with a point of order $m$ for $m=4$ or $8$ includes curves with full $2$-torsion. Since the torsion points of $E(\mathbb{Q})$ injects into $E_p(\mathbb{F}_p)$, the reductions of these curves never have cyclic group of $\mathbb{F}_p$-rational points. It would be interesting to obtain a similar result where the average is taken only over curves with cyclic 2-torsion. Furthermore, we note that we have not treated all torsion families; instead we have focused on curves in one-parameter families. The two-parameter families (curves with 2-torsion or 3-torsion points) will be the subject of future work.
\end{remark}

\begin{remark} The proof broadly follows the steps of \cite{banksshparlinski}; however, we note that an important feature of that work was the use of character sums to reduce the size of the family over which the average is taken. The torsion families over which we average in the current paper are each a one-parameter family, and it is unclear how to adapt the character sum estimates to this context. Consequently, we average over curves in a larger family than we would prefer. Reducing the size of the family will be the subject of subsequent work. 
\end{remark}

 A key ingredient for the argument in \cite{banksshparlinski} was the fixed-field count of Vl\u{a}du\c{t} \cite{vladut} which, for a finite field of $q$ elements, estimates the number of $E/\mathbb{F}_q$ which have cyclic group structure. Our proof requires the following analogous fixed field count which takes into account the additional torsion data.

It is frequently convenient to express counts of elliptic curves over finite fields as weighted cardinalities where we weight each curve by the size of its automorphism group.  We indicate weighted cardinalities by $\#'$.

For a prime number $\ell$, denote by $v_\ell(n)$ the $\ell$-adic valuation of $n$. Concretely, any positive integer $n$ can be written $n=\ell^em$ where $\ell\nmid m$. Then $v_\ell(n)=e$.
 
\begin{theorem}\label{fixedfield} Denote by \[
C_q(m)=\{E/\mathbb{F}_q: E(\mathbb{F}_q) \text{ is cyclic and contains a point of order $m$}\}/_{\cong\mathbb{F}_q}.\]
Then
\[\#'C_q(m)=q\hspace{-16pt}\prod_{\substack{\ell|m\\ q\equiv 1\pmod*{\ell}}}\frac{1}{\ell^{v_\ell(m)}}\prod_{\substack{\ell|m\\ q\not\equiv 1\pmod*{\ell}}}\frac{1}{\varphi(\ell^{v_\ell(m)})}\prod_{\substack{\ell\nmid m\\ q\equiv 1\pmod*{\ell} }}\left(1-\frac{1}{\ell(\ell^2-1)}\right)+O\left(q^{1/2}\right).\]

\end{theorem}
The structure of this paper is as follows. In Section 2, we give the proof of Theorem \ref{fixedfield}. In Section 3, we describe the isomorphism classes of $E/\mathbb{F}_p$ in the torsion families described above. In Section 4, we give the proof of Theorem \ref{mainresult}.

\section{Fixed Field Counts}
There has been significant recent interest in counting problems for elliptic curves over a fixed finite field $\mathbb{F}_q$ with specified conditions on their group of $\mathbb{F}_q$-rational points. See for example Howe, \cite{howeMR1204781}, Vl\u{a}du\c{t} \cite{vladut}, Castryck and Hubrechts \cite{castryckhubrechts} and Kaplan and Petrow \cite{kaplanpetrow}.

Our goal is to generalize Vl\u{a}du\c{t}'s result giving the number of elliptic curves $E/\mathbb{F}_q$ such that $E(\mathbb{F}_q)$ is cyclic to obtain a count of the number of $E/\mathbb{F}_q$ such that for some fixed $m\in \mathbb{Z}$, $E(\mathbb{F}_q)$ has a point of order $m$ and is cyclic. We begin by recalling the various fixed field counts we will require.

Denote by $\varphi$ the Euler totient function, and define $\psi(n)=n\prod_{l\mid n}(1+1/l)$. For $a\mid b$, denote by $W(a, b)=\{E/\mathbb{F}_q: E[b](\mathbb{F}_q)\cong \left(\mathbb{Z}/a\mathbb{Z}\right)\times\left(\mathbb{Z}/b\mathbb{Z}\right)\}/_{\cong\mathbb{F}_q}$. Estimates for the size of $W(a,b)$ are given by Howe \cite{howeMR1204781}. Howe shows that 
\[ \mid \#W(a, b)-\hat{w}(a,b)\mid<Cq^{1/2} 
\]
for an explicit constant $C$ where 
\[\hat{w}(a, b)= \frac{q\psi(b/a)}{a\varphi(b)\psi(b)}\prod_{\substack{l \text{ prime}\\ l\mid \gcd(b, q-1)/b}}\left(1-\frac{1}{l}\right).\]
It will be convenient to define $\tilde{w}(a,b)=\hat{w}(a,b)/q$. Howe notes that $\tilde{w}(a, b)$ is a multiplicative function of both arguments simultaneously.

Vl\u{a}du\c{t} observes the `obvious' cyclicity condition: $E(\mathbb{F}_q)$ is cyclic if and only if for any prime $l\mid q-1$, $E \not\in W(l,l)$. Assuming that our curve $E$ has a point of order $m$, we observe that $E(\mathbb{F}_q)$ is cyclic if and only if for all $l\mid q-1$, 
\[
\begin{cases} 
E \not\in W(l,m) \text{ (in the case where $l\mid m$)},\\
E \not\in W(l,lm) \text{ (in the case where $l\nmid m$)}.
\end{cases}
\]
For $q=p^n$, define $r_q'(m)$ by the following conditions:
\begin{enumerate}[(a)]
\item $r_q'$ is multiplicative.
\item For $l$ prime, $l\neq p$, we have 
\[ r_q'(l^n)=\begin{cases} 1/(l^n-l^{n-2}) &\text{ if } v\geq n, \\
(l^{2v+1}+1)/(l^{n+2v-1}(l^2-1)) &\text{ if } v<n,
\end{cases}
\]
where $v=v_l(q-1)$.
\item $r_q'(p^e)=1/(p^e-p^{e-1})$.
\end{enumerate}
Denote by $T_q(m)=\{E/\mathbb{F}_q: E\text{ has a point of order $m$}\}/_{\cong\mathbb{F}_q}$. It follows from Theorem 3 of \cite{castryckhubrechts} that 
\[\mid\#T_q(m)-qr'(m)\mid\leq Cm^2\log\log(m)q^{1/2}\]
for an absolute and explicitly computable $C\in \mathbb{R}_{>0}$.

Applying these estimates and the inclusion/exclusion principle, we obtain the estimate
\begin{equation} \label{eq:1}
\#'C_q(m)=qr'_q(m)-\sum_{\substack{d\mid m \\ d>1}}\hat{w}(d, m)+\sum_{\substack{t\mid q-1\\ \gcd(t, m)=1}}\mu(t)\hat{w}(t, mt)+O\left(q^{1/2}\right).
\end{equation}
Note that we have
\[\tilde{w}(1, m)=\prod_{l\mid m}\tilde{w}(1, l^{v_l(m)})=\prod_{\substack{l\mid m\\v_l(q-1)=0}}\frac{1}{\varphi(l^{v_l(m)})}\prod_{\substack{l\mid m\\v_l(q-1)>0}}\frac{1}{l^{v_l(m)}}.\] 

Equating the right hand side of \eqref{eq:1} to the right hand side of the equation in Theorem \ref{fixedfield} and dividing through by $q$, we see that it is sufficient to show that
\[
r'_q(m)-\sum_{\substack{d\mid m \\ d>1}}\tilde{w}(d, m)+\sum_{\substack{t\mid q-1\\ \gcd(t, m)=1}}\mu(t)\tilde{w}(t, mt)=\tilde{w}(1, m)\prod_{\substack{\ell|q-1\\ \ell\nmid m}}\left(1-\frac{1}{\ell(\ell^2-1)}\right)+O\left(q^{-1/2}\right). 
\]
Furthermore, since $\tilde{w}(t, mt)=\tilde{w}(t,t)\tilde{w}(1,m)$ for $t$ relatively prime to $m$, it suffices to show that 
\begin{equation} \label{eq:2}
r_q'(m)=\sum_{\substack{d\mid m}}\tilde{w}(d, m). 
\end{equation}
Indeed, in this case, we have 
\begin{eqnarray*}
C_q(m) &=&qr'_q(m)-\sum_{\substack{d\mid m\\ d>1}} \hat{w}(d, m)+\sum_{\substack{t\mid q-1\\ \gcd(t, m)=1}}\mu(t)\hat{w}(t, mt)+O(q^{1/2})\\
&=& qr_q'(m)-q\sum_{\substack{d\mid m\\ d>1}} \tilde{w}(d, m)+q\sum_{\substack{t\mid q-1\\ \gcd(t, m)=1}}\mu(t)\tilde{w}(t, t)\tilde{w}(1, m)+O(q^{1/2}))\\
&=&q\tilde{w}(1, m)\left(1+\sum_{\substack{t\mid q-1\\ \gcd(t, m)=1}}\frac{\mu{t}}{t\varphi(t)\psi(t)}\right)+O(q^{1/2})\\
&=& q\prod_{\substack{l\mid m\\v_l(q-1)=0}}\frac{1}{\varphi(l^{v_l(m)})}\prod_{\substack{l\mid m\\v_l(q-1)>0}}\frac{1}{l^{v_l(m)}}\prod_{\substack{l\mid q-1\\ \gcd(m, l)=1}}\left(1-\frac{1}{l(l^2-1)}\right)+O(q^{1/2}).
\end{eqnarray*}

We will prove \eqref{eq:2} by induction on the number of prime factors of $m$.
The following lemma provides the base of induction.
\begin{lemma}\label{basecase} Denote by $v$ the $\ell$-adic valuation of $q-1$. Then
\begin{eqnarray}
r_q'(\ell^e)=\sum_{k=0}^t\tilde{w}(\ell^k, \ell^e)
\end{eqnarray}
where 
\begin{eqnarray*}
t=\begin{cases} v &\text{ if } v<e\\
				e & \text{ if } v\geq e.
				\end{cases}
\end{eqnarray*}
\end{lemma}
\begin{proof}
Suppose that $v=0$. Then 
\begin{eqnarray*}
\sum_{k=0}^v \tilde{w}(\ell^k, \ell^e) =\tilde{w}(1, \ell^e) =\frac{\psi(\ell^e)}{\varphi(\ell^e)\psi(\ell^e)}=r_q'(\ell^e).
\end{eqnarray*}

Now suppose that $0<v<e$. Then we have

\begin{eqnarray*}
\sum_{k=0}^v\tilde{w}(\ell^k, \ell^e) &=& \frac{\psi(\ell^{e-v})}{\ell^v\varphi(\ell^e)\psi(\ell^e)} + \sum_{k=0}^{v-1}\frac{\psi(\ell^{e-k})(\ell-1)}{\ell^{k+1}\varphi(\ell^e)\psi(\ell^e)} \\
&=& \frac{\ell^{e-v-1}(\ell+1)}{\ell^v\varphi(\ell^e)\psi(\ell^e)} + \sum_{k=0}^{v-1}\frac{\ell^{e-k-1}(\ell^2-1)}{\ell^{k+1}\varphi(\ell^e)\psi(\ell^e)} \\
&=& \frac{1}{\ell^e\varphi(\ell^e)\psi(\ell^e)}\left( \ell^{2e-2v}+\ell^{2e-2v-1}+ \sum_{k=0}^{v-1}(\ell^{2e-2k}-\ell^{2e-2k-2}) \right).
\end{eqnarray*}

The sum above telescopes to $\ell^{2e}-\ell^{2e-2v}$, and we are left with
\begin{eqnarray*}
 \frac{ \ell^{2e}+\ell^{2e-2v-1}}{\ell^e\varphi(\ell^e)\psi(\ell^e)}
&=& \frac{\ell^{2e-2v-1}(\ell^{2v+1}+1)}{\ell^{3e-2}(\ell^2-1)}\\
&=&\frac{\ell^{2v+1}+1}{\ell^{e-2v-1}(\ell^2-1)} =r_q'(\ell^{e}),
\end{eqnarray*}

as needed.

Finally, suppose $v\geq e$. we have 
\begin{eqnarray*}
\sum_{k=0}^e\tilde{w}(\ell^k, \ell^e)&=&\frac{1}{\ell^e\varphi(\ell^e)\psi(\ell^e)} + \sum_{k=0}^{e-1}\frac{\psi(\ell^{e-k})(\ell-1)}{\ell^{k+1}\varphi(\ell^e)\psi(\ell^e)}\\ &=& \frac{1}{\ell^e\varphi(\ell^e)\psi(\ell^e)}+\sum_{k=0}^{e-1}\frac{\ell^{e-k-1}(\ell^2-1)}{\ell^{k+1}\varphi(\ell^e)\psi(\ell^e)} \\ 
&=& \frac{1}{\ell^e\varphi(\ell^e)\psi(\ell^e)}\left(1 + \sum_{k=0}^{e-1}(\ell^{2e-2k}-\ell^{2e-2k-2}) \right).
\end{eqnarray*}
As above, this sum telescopes, leaving $\ell^{2e}-1$. We are left with
\begin{eqnarray*}
\frac{\ell^{2e}}{\ell^e\varphi(\ell^e)\psi(\ell^e)}=\frac{\ell^{2e}}{\ell^{3e-2}(\ell^2-1)}=\frac{1}{\ell^e-\ell^{e-2}}=r_q'(\ell^e).
\end{eqnarray*}

\end{proof}

\begin{proof}[Proof of Theorem \ref{fixedfield}]
Assume for induction that for some $m>1$ we have
\begin{eqnarray*} 
r_q'(m)=\sum_{\substack{d\mid m}} \tilde{w}(d, m).
\end{eqnarray*}

Let $\ell$ be a prime that does not divide $m$ and let $e\geq 1$. Let $t$ be as in the statement of Lemma \ref{basecase}. We have
 
\begin{eqnarray*}
\sum_{\substack{d\mid m\ell^e}} \tilde{w}(d, m\ell^e)=\sum_{k=0}^t \sum_{\substack{d\mid m}}\tilde{w}(\ell^k d, \ell^e m)=\sum_{k=0}^t\tilde{w}(\ell^k, \ell^e)\sum_{\substack{d\mid m}}\tilde{w}(d, m)= r_q'(\ell^e) r_q'(m)=r_q'(\ell^em)
\end{eqnarray*} 
by Lemma \ref{basecase} and the induction hypothesis.

\end{proof}

\section{Isomorphisms of $m$-torsion Curves}
\iffalse 
This was the first draft of the section; I am going to rewrite it to not refer to battistabayless since I think that there is an error there, or possibly it is correct but is written in a way that doesn't explicitly address E with Aut E large.

As noted in \cite{battistabayless}, for $m\in \{5,7,9\}$, if $E/\mathbb{F}_p$ is cyclic, then there are $\varphi(m)/2$ parameter values $\pmod{p}$ which yield isomorphic curves. In fact, it is not hard to see that for any $m\in\{3,5,6,7,8,9,10,12\}$,  $\varphi(m)/2$ parameters $\pmod{p}$ give isomorphic curves since these parameterizations are derived from the modular curve $X_1(m)$, whose points parameterize an elliptic curve together with a point of order $m$ up to action by automorphisms of $E$. Since for any elliptic curve, $P\mapsto -P$ is an automorphism, of the $\varphi(m) $ points of order $m$ in the subgroup generated by $m$, at most $\varphi(m)/2$ of them result in distinct points of $X_1(m)$. When $p>3$, with the exception of at most ten isomorphism classes, the size of the automorphism group of $E$ is $2$ (see \cite{lenstrafactoringintegers}). Thus, each isomorphism class corresponds to $\varphi(m)/2+O(1)$ values of $a\pmod{p}$.  Using explicit change of variables, it is possible to specify precisely which parameter values yield isomorphic curves. We summarize this below for curves with $\#\Aut(E)=2$.
\fi

The parameterizations given in Table \ref{tab:table-parameterizations} were derived by Kubert by studying the modular curve $X_1(m)$. A point of $X_1(m)$ corresponds to an elliptic curve $E$ together with a point of order $m$ up to action by automorphisms of $E$. For example, if $(E, P)$ represents a point of $X_1(m)$ and $\#\text{Aut}(E)=2$, then $(E, -P)$ represents the same point. We will be concerned with which parameter values yield isomorphic curves over $\mathbb{F}_q$ where $q=p^n$ and $p>3$. With at most 10 exceptions, an isomorphism class of elliptic curves Over $\mathbb{F}_q$ consists of curves whose automorphism group has cardinality 2. If $\#\Aut(E)=2$ and $E$ has cyclic $m$-torsion, then there are $\varphi(m)/2$ parameters in $\mathbb{F}_q$ that yield an isomorphic curve. These correspond to the $\varphi(m)$ points of order $m$ on $E_m(a)$, up to the action of of $\Aut(E_m(a))$. 

If $\#\Aut(E)>2$, the number of parameters yielding an isomorphic curve will vary depending on the size of the automorpism group and the number of $m$-torsion points of $E$; in general, the number of parameters which yield an isomorphic curve will not be $\varphi(m)/2$. However, these $O(1)$ isomorphism classes can be absorbed into the `unweighted' version of Theorem \ref{fixedfield}.
 
\begin{corollary}[to Theorem \ref{fixedfield}]\label{unweightedfixedfield}
\begin{eqnarray*}
C_q(m)=2q\hspace{-16pt}\prod_{\substack{\ell|m\\ q\equiv 1\pmod*{\ell}}}\frac{1}{\ell^{v_\ell(m)}}\prod_{\substack{\ell|m\\ q\not\equiv 1\pmod*{\ell}}}\frac{1}{\varphi(\ell^{v_\ell(m)})}\prod_{\substack{\ell\nmid m\\ q\equiv 1\pmod*{\ell} }}\left(1-\frac{1}{\ell(\ell^2-1)}\right)+O\left(q^{1/2}\right).
\end{eqnarray*}
\end{corollary}
\begin{proof}
Denote by $C_q(m, n)=\{E\in C_q(m)\colon \#\Aut(E)=n\}$. Then
\begin{eqnarray*}
C_q(m)=C_q(m, 2)+C_q(m, 4)+C_q(m, 6).
\end{eqnarray*}
We then have
\begin{eqnarray*}
\#C_q(m, 2)/2=\#'C_q(m)-\#C_q(m, 4)/4\#C_q(m, 6)/6.
\end{eqnarray*}
Multiplying by 2 and observing that $\#C_q(m, 4)/2+\#C_q(m, 6)/3=O(1)$, we have \begin{eqnarray*}
\#C_q(m, 2)=2q\hspace{-16pt}\prod_{\substack{\ell|m\\ q\equiv 1\pmod*{\ell}}}\frac{1}{\ell^{v_\ell(m)}}\prod_{\substack{\ell|m\\ q\not\equiv 1\pmod*{\ell}}}\frac{1}{\varphi(\ell^{v_\ell(m)})}\prod_{\substack{\ell\nmid m\\ q\equiv 1\pmod*{\ell} }}\left(1-\frac{1}{\ell(\ell^2-1)}\right)+O\left(q^{1/2}\right)
\end{eqnarray*}
as required.
\end{proof}

Given an $m$-torsion curve $E_m(a)$, we can perform a change of variables to obtain a short Weierstrass equation $y^2=x^3+Ax+B$. In order for $E_m(a)$ to have more than two automorphisms, the $j$-invariant must be $0$ or 1728. In terms of the short Weierstass equation, this means that $A=0$ or $B=0$, respectively. The coefficients $A$ and $B$ will be polynomials or rational functions in the parameter $a$. Since such functions have finitely many zeros, we deduce the following

\begin{lemma} \label{extraautbound}
Each torsion family $E_m(a)$ contains finitely many curves with $j$-invariant $0$ or $1728$. The parameters which yield curves with these $j$-invariants are roots of a polynomial that depends only on $m$. 
\end{lemma}

Using explicit change of variables, it is possible to specify precisely which parameter values yield isomorphic curves. We summarize this below for curves with $\#\Aut(E)=2$.
\begin{table}[ht]
\centering
\begin{tabular}{|c|c|c|c| }
  \hline
$m$& \multicolumn{3}{|c|}{ Parameters } \\
  \hline
  4 &  \multicolumn{3}{|c|}{$a$} \\
\hline
 5 &\multicolumn{2}{|c|}{$a$} &\multicolumn{1}{|c|}{$-a^{-1}$}\\
\hline
  6 &\multicolumn{3}{|c|}{$a$}   \\
\hline
  7 & $a$  &$(1-a)a^{-1}$ & $-(1-a)^{-1}$ \\
\hline
  8 & $a$  &\multicolumn{2}{|c|}{ $-a+1$ }    \\
\hline
  9 &   $a$&  $(a-1)a^{-1}$  & $-(a-1)^{-1} $   \\
\hline
  10 & $a$&\multicolumn{2}{|c|}{  $(a-1)(2a-1)^{-1}$}    \\
\hline
  12 & $a$ & \multicolumn{2}{|c|}{ $-a+1$}\\
  \hline
\end{tabular}
\caption{Parameters yielding isomorphic curves. \label{tab:table-parms}}
\end{table}
\FloatBarrier

%The number of parameters yielding isomorphic curves is given by $\varphi(m)/2$ in each case.

\section{Proof of Theorem \ref{mainresult}}
The fixed field counts of $E/\mathbb{F}_p$ which have an $m$-torsion point and cyclic group of $\mathbb{F}_p$-points depends on the value of $p$ modulo the prime divisors of $m$. For one-parameter torsion families over $\mathbb{Q}$, we are concerned with $m\in \{4,5,6,7,8,9,10, 12\}$. In this case, $m$ has at most two prime divisors, and if $m$ is not a prime power, then one of its prime factors is 2. Since all odd primes are $1\pmod{2}$, the number of curves we are counting varies according to the value of $p\pmod{\ell}$ where $\ell$ is the unique odd prime divisor of $m$.

Let $m\in \{4,5,6,7,8,9,10,12\}$, and denote by 
\begin{eqnarray*}
\mathcal{E}_m(A)=\{E_m(a):-A\leq a \leq A\}
\end{eqnarray*}
the family of elliptic curves over $\mathbb{Q}$ with an $m$-torsion point given above. Write the prime factorization of $m$ as $m=2^k\ell_0^n$ where we take $n=\ell_0=1$ if $m$ is a power of 2. Let $A\geq x^{1+\epsilon}$ for $x, \epsilon \geq 0$.  

We have
\begin{eqnarray*}
\sum_{|a|\leq A} \pi_{E_m(a)}^{cyc}(x)=\sum_{p\leq x}\sum_{\substack{b\in\mathbb{F}_p\\ \Delta_m(b)\neq 0\\ E_m(b)(\mathbb{F}_p) \text{ cyclic }}}\#\{a\in [-A, A]: E_m(a)_p\cong E_m(b)\}. 
\end{eqnarray*} There are $2A/p+O(1)$ values of $a$ which yield a particular Weierstrass model modulo $p$. Thus the expression above becomes
\begin{eqnarray*}
\sum _{p\leq x}\left(\left(\frac{2A}{p}+O(1)\right)\left(
\frac{\varphi(m)}{2}\hspace{-25pt}\sum_{\substack{b\in\mathbb{F}_p\\ \Delta_m(b)\neq 0\\ E_m(b)(\mathbb{F}_p) \text{ cyclic }\\ \#\Aut(E_m(b))=2}}1\hspace{25pt}+ 
\frac{\varphi(m)}{4}\hspace{-25pt}\sum_{\substack{b\in\mathbb{F}_p\\ \Delta_m(b)\neq 0\\ E_m(b)(\mathbb{F}_p) \text{ cyclic }\\ \#\Aut(E_m(b))=4}}1\hspace{25pt}+
\frac{\varphi(m)}{6}\hspace{-25pt}\sum_{\substack{b\in\mathbb{F}_p\\ \Delta_m(b)\neq 0\\ E_m(b)(\mathbb{F}_p) \text{ cyclic }\\ \#\Aut(E_m(b))=6}}1 \right)\right). 
\end{eqnarray*}

Applying the estimate obtained in Corollary \ref{unweightedfixedfield} and Lemma \ref{extraautbound}, this is equal to
\begin{eqnarray*}
\sum _{\substack{p\leq x\\ p\equiv 1\pmod*{\ell_0}}}\left(\frac{2A}{p}+O(1)\right)\left(\frac{\varphi(m)}{2}\frac{2p}{m}\prod_{\substack{\ell|p-1\\ \gcd(\ell, m)=1}}\left(1-\frac{1}{\ell(\ell^2-1)}\right)+O\left(p^{1/2}\right)\right)\\ 
+\sum _{\substack{p\leq x\\ p\not\equiv 1\pmod*{\ell_0}}}\left(\frac{2A}{p}+O(1)\right)\left(\frac{\varphi(m)}{2}\frac{2p}{2^k\varphi(\ell_0^n)}\prod_{\substack{\ell|p-1\\ \gcd(\ell, m)=1}}\left(1-\frac{1}{\ell(\ell^2-1)}\right)+O\left(p^{1/2}\right)\right).
\end{eqnarray*}
Simplifying and applying the trivial estimate \[\frac{\varphi(m)p}{2^k\varphi(\ell_0^n)}\prod_{\substack{\ell|p-1\\ \gcd(\ell, m)=1}}\left(1-\frac{1}{\ell(\ell^2-1)}\right)<p,\] we have
\begin{align}\label{sumoverprimes}
\sum _{\substack{p\leq x\\ p\equiv 1\pmod*{\ell_0}}} \frac{2A\varphi(m)}{m}  \prod_{\substack{\ell|p-1\\ \gcd(\ell, m)=1}}\left(1-\frac{1}{\ell(\ell^2-1)}\right)\nonumber \\ 
+\sum _{\substack{p\leq x\\ p\not\equiv 1\pmod*{\ell_0}}} \frac{2A\varphi(m)}{2^k\varphi(\ell_0^n)}  \prod_{\substack{\ell|p-1}}\left(1-\frac{1}{\ell(\ell^2-1)}\right) +O\left(\sum_{p\leq x}p\right).
\end{align}

Note that according to Lemma 3.4 of \cite{trevinoleastinertprime}, we have
\begin{eqnarray*}
\sum_{p\leq x}p = O\left(\frac{x^2}{2\log{x}}\right).
\end{eqnarray*}

 We analyze these two sums individually following \cite{banksshparlinski}. A main input to this analysis is a theorem on averages of multiplicative functions due to \cite[Theorem 3] {indlekofermeanbehaviour}.

Assume first that $m$ is not a power of $2$, so that $\ell_0>1$.
For any integer $n$, define 
\begin{eqnarray*}
\chi_{\ell_0}(n)=\begin{cases}
1 \text{ if } \ell_0\nmid n\\
0 \text{ if } \ell_0 |n,
\end{cases}
\end{eqnarray*}

\[ F(n) = \prod_{\substack{\ell|n\\ \ell\nmid m }}\left(1-\frac{1}{\ell(\ell^2-1)}\right),\]
and 
\[ F'(n) = \prod_{\substack{\ell|n\\ \ell\nmid m }}\left(1-\frac{1}{\ell(\ell^2-1)}\right)\chi_{\ell_0h}(n).\]
Note that $\chi_{\ell_0}$ and $F$ are multiplicative (whence $ F'$ is multiplicative as well). We compute

\begin{eqnarray*}
&&\sum _{\substack{p\leq x\\ p\equiv 1\pmod*{\ell_0}}}\prod_{\substack{\ell|p-1\\ \gcd(\ell, m)=1}}\left(1-\frac{1}{\ell(\ell^2-1)}\right)  \\
&=&\sum _{\substack{p\leq x}}   \prod_{\substack{\ell|p-1\\ \gcd(\ell, m)=1}}\left(1-\frac{1}{\ell(\ell^2-1)}\right)-\sum _{\substack{p\leq x\\ p\not\equiv 1\pmod*{\ell_0}}}   \prod_{\substack{\ell|p-1\\ \gcd(\ell, m)=1}}\left(1-\frac{1}{q(q^2-1)}\right)\\
&=&\sum _{\substack{p\leq x}}   \prod_{\substack{\ell|p-1\\ \gcd(\ell, m)=1}}\left(1-\frac{1}{\ell(\ell^2-1)}\right)-\sum _{\substack{p\leq x}}   \prod_{\substack{\ell|p-1\\ \gcd(\ell, m)=1}}\left(1-\frac{1}{\ell(\ell^2-1)}\right)\chi_{\ell_0}(p-1)\\
&=&\sum _{\substack{p\leq x}} F(p-1) +\sum _{\substack{p\leq x}} F'(p-1).
\end{eqnarray*}
 Denote by $G$ and $G'$ the Dirichlet convolution of the M\"{o}bius $\mu$ function with $F$, $F'$, respectively. Explicitly, $G$ and $G'$ are multiplicative functions defined on prime powers by
\[G(\ell^k)=\begin{cases}
\frac{-1}{\ell(\ell^2-1)} &\text{ if } \ell\nmid m, k=1\\
0                               &\text{ if } \ell\mid m, k=1\\
0                               &\text{ if } k>1,
\end{cases}
\]
and
\[G'(\ell^k)=\begin{cases}
\frac{-1}{\ell(\ell^2-1)} &\text{ if } \ell\nmid m, k=1\\
-1                               &\text{ if } \ell=\ell_0, k=1\\
0                               &\text{ if } \ell= 2\mid m, k=1\\
0                               &\text{ if } k>1.
\end{cases}
\]
Both pairs of functions $F, G$ and $F', G'$ satisfy the hypotheses of  \cite[Theorem 3]{indlekofermeanbehaviour}. It follows that
%\[\lim_{x\to\infty}\frac{1}{\pi(x)}\sum_{p\leq \pi(x)}F(p-1)=\sum_{d=1}^\infty\frac{G(d)}{\varphi(d)}\]
\[\frac{1}{\pi(x)}\sum_{p\leq x}F(p-1)=\sum_{d=1}^\infty\frac{G(d)}{\varphi(d)}+O_B(\log^{-B}x)\]
holds for any $B>0$, and similarly for $F', G'$. Now we have
\[\sum_{d=1}^\infty\frac{G(d)}{\varphi(d)}=\prod_{\ell\nmid m}\left(1-\frac{1}{\ell(\ell^2-1)(\ell-1)}\right),\]
and
\[\sum_{d=1}^\infty\frac{G'(d)}{\varphi(d)}=\frac{\ell_0-2}{\ell_0-1}\prod_{\ell\nmid m}\left(1-\frac{1}{\ell(\ell^2-1)(\ell-1)}\right).\]
Thus, 

\begin{eqnarray}\label{estfor1modl}
\sum _{\substack{p\leq x\\ p\equiv 1\pmod*{\ell_0}}} \frac{2A\varphi(m)}{m}  \prod_{\substack{\ell|p-1\\ \gcd(\ell, m)=1}}\left(1-\frac{1}{\ell(\ell^2-1)}\right)\nonumber \\ =\frac{2A\varphi(m)}{m}\frac{1}{\ell_0-1}\prod_{\ell\nmid m}\left(1-\frac{1}{\ell(\ell-1)(\ell^2-1)}\right)\pi(x)+O\left(A\frac{x}{\log^{B+1}}\right).
\end{eqnarray}

Similarly, we write
\begin{align*}
 \sum _{\substack{p\leq x\\ p\not\equiv 1\pmod*{\ell_0}}}  \prod_{\substack{\ell|p-1}}\left(1-\frac{1}{\ell(\ell^2-1)}\right)= \sum _{\substack{p\leq x}}  \prod_{\substack{\ell\mid p-1}}\left(1-\frac{1}{\ell(\ell^2-1)}\right)\chi_{\ell_0}(p-1)= \sum _{p\leq x}F'(p-1)
\end{align*}
so that, again by \cite[Theorem 3]{indlekofermeanbehaviour}, we have
\begin{align*}\frac{1}{\pi(x)}\sum_{p\leq x}F'(p-1)=\sum_{d=1}^\infty\frac{G'(d)}{\varphi(d)}+O_B(\log^{-B}x)\nonumber \\ =\frac{\ell_0-2}{\ell_0-1}\prod_{\ell\nmid 2\ell_0}\left( 1-\frac{1}{\ell(\ell-1)(\ell^2-1)}\right)+O_B(\log^{-B}x)
\end{align*}
holds for any $B>0$.
Thus,  
\begin{align}\label{estfornot1modl}
\sum _{\substack{p\leq x\\ p\not\equiv 1\pmod*{\ell_0}}} \frac{2A\varphi(m)}{2^k\varphi(\ell_0^n)}  \prod_{\substack{\ell|p-1}}\left(1-\frac{1}{\ell(\ell^2-1)}\right)\nonumber \\=\frac{2A\varphi(m)}{2^k\varphi(\ell_0^n)}\frac{\ell_0-2}{\ell_0-1}\prod_{\ell\nmid m}\left( 1-\frac{1}{\ell(\ell-1)(\ell^2-1)}\right)\pi(x)+O\left(A\frac{x}{\log^{B+1}}\right).
\end{align}

Combining \eqref{sumoverprimes}, \eqref{estfor1modl}, and \eqref{estfornot1modl}, we have 
\begin{eqnarray*}
\sum_{|a|\leq A} \pi_{E(a)}^{cyc}(x) = \left( \frac{2A\varphi(m)}{2^k\varphi(\ell_0^n)}\frac{\ell_0-2}{\ell_0-1}+\frac{2A\varphi(m)}{m}\frac{1}{\ell_0-1}\right)\prod_{\ell\nmid m}\left(1-\frac{1}{\ell(\ell-1)(\ell^2-1)}\right)\pi(x) \\ +O\left(A\frac{x}{\log^{B+1}}\right)+O\left(\frac{x^2}{2\log^{x}}\right).
\end{eqnarray*}

If $m$ is a power of $2$, the sum over $p\not \equiv 1 \pmod{\ell_0}$ is empty. In this case, $\varphi(m)/m=1/2$, and we are left to evaluate 
\[\frac{1}{2}\sum _{\substack{p\leq x}}   \prod_{\substack{\ell|p-1\\ \gcd(\ell, m)=1}}\left(1-\frac{1}{\ell(\ell^2-1)}\right).\]

An argument analogous to the above shows
\[
\sum_{|a|\leq A} \pi_{E(a)}^{cyc}(x) =A\prod_{\ell\neq 2}\left(1-\frac{1}{\ell(\ell-1)(\ell^2-1)}\right)\pi(x)  +O\left(A\frac{x}{\log^{B+1}}\right)+O\left(\frac{x^2}{2\log^{x}}\right).\]

Computing 

\begin{align*}
C_m=\left( \frac{2A\varphi(m)}{2^k\varphi(\ell_0^n)}\frac{\ell_0-2}{\ell_0-1}+\frac{2A\varphi(m)}{m}\frac{1}{\ell_0-1}\right) 
\end{align*} for $m\in\{4,5,6,7,8,9,10,12\}$,we complete the proof.

%\begin{remark}
%We have averaged over elliptic curves with given torsion by varying over a one-parameter family of curves. One could also study this average by taking elliptic curves given by short Weierstrass equations
%\end{remark}

\bibliographystyle{plain}

\bibliography{bibtex.bib}

\end{document}